%% file: main.tex
\documentclass[11pt,final,twoside]{article}
\usepackage[textwidth=6in,textheight=8.5in,centering]{geometry}%
\usepackage{titlesec,titletoc}
\usepackage{etex}
\usepackage{mathtools}
\usepackage{pdfpages}
\usepackage{url}
\usepackage{amsmath}
\usepackage{amsxtra,amscd}
\usepackage{amsfonts}
\usepackage{amssymb}
\usepackage{graphicx}
\usepackage[amsmath,hyperref,thmmarks]{ntheorem}
\usepackage{setspace}
\usepackage[backend=bibtex,style=trad-alpha,giveninits=true,sorting=nty,doi=false,isbn=false,url=false,eprint=false]{biblatex}
\usepackage{verbatim}
\usepackage{mathptmx}
\usepackage{enumitem}
\usepackage[notcite,notref]{showkeys}
\usepackage{url}
\usepackage[colorlinks=true,bookmarksnumbered=true,pdfpagemode=None,extension=pdf]%
{hyperref}%
\providecommand{\U}[1]{\protect\rule{.1in}{.1in}}
\theoremnumbering{arabic}
\theoremheaderfont{\scshape}
\RequirePackage{latexsym}
\theorembodyfont{\slshape}
\theoremseparator{.}
\newtheorem{X}{X}[section]

\newtheorem{conjecture}[X]{Conjecture}
\newtheorem{corollary}[X]{Corollary}

\newtheorem{lemma}[X]{Lemma}
\newtheorem*{*lemma}{Lemma}
\newtheorem{proposition}[X]{Proposition}
\newtheorem*{*proposition}{Proposition}

\newtheorem{theorem}[X]{Theorem}
\theorembodyfont{\upshape}

\newtheorem*{*definition}{Definition}

\newtheorem*{*example}{Example}

\newtheorem{question}[X]{Question}
\newtheorem{remark}[X]{Remark}

\newtheorem{hypothesis}[X]{Hypothesis}
\theorembodyfont{\footnotesize}

\theorembodyfont{\normalsize}
\theoremstyle{nonumberplain}
\theoremheaderfont{\sc}
\theorembodyfont{\normalfont}
\theoremsymbol{\ensuremath{_\Box}}
\RequirePackage{amssymb}
\newtheorem{proof}{Proof}
\qedsymbol{\ensuremath{_\Box}}
\theoremclass{LaTeX}
\makeindex
\input{defa}
\input{defs}

\input{defs-extra}
\begin{document}

\title{The Distribution of G.C.D.s of Shifted Primes and Lucas Sequences}
\author{Abhishek Jha and Ayan Nath}
\date{\vspace{-5ex}}
\maketitle

\begin{abstract}
    Let $(u_n)_{n \ge 0}$ be a nondegenerate Lucas sequence
    and $g_u(n)$ be the arithmetic function defined by $\gcd(n, u_n).$
    Recent studies have investigated the distributional characteristics of $g_u$.
    Numerous results have been proven based on the two extreme values $1$ and $n$ of $g_{u}(n)$. 
    Sanna investigated the average behaviour of $g_{u}$ and found asymptotic formulas for the moments of $\log g_{u}$.
    In a related direction, Jha and Sanna investigated properties of $g_{u}$ at shifted primes.
    
    In light of these results,
    we prove that for each positive integer $\lambda,$ we have 
    \begin{equation*}
    \sum_{\substack{p\le x\\p\text{ prime}}} (\log g_{u}(p-1))^{\lambda} \sim P_{u,\lambda}\pi(x),\end{equation*}
    where $P_{u, \lambda}$ is a constant depending on $u$ and $\lambda$ which is expressible as an infinite series. 
    Additionally, we provide estimates for $P_{u,\lambda}$ and $M_{u,\lambda},$ where $M_{u, \lambda}$ is the constant for an analogous sum obtained by Sanna [J. Number Theory \textbf{191} (2018), 305--315]. As an application of our results, we prove upper bounds on the count $\#\{p\le x : g_{u}(p-1)>y\}$ and also establish the existence of infinitely many runs of $m$ consecutive primes $p$ in bounded intervals such that $g_{u}(p-1)>y$ based on a breakthrough of Zhang, Maynard, Tao, et al. on small gaps between primes.
    Exploring further in this direction, it turns out that for Lucas sequences with nonunit discriminant, we have $\max\{g_{u}(n) : n \le x\} \gg x$.
    As an analogue, we obtain that that $\max\{g_u(p-1) : p \le x\} \gg x^{0.4736}$ unconditionally, while $\max\{g_u(p-1): p \le x\} \gg x^{1 - o(1)}$ under the hypothesis of Montgomery's or Chowla's conjecture.
\end{abstract}

\tableoc

\section{Introduction}
Let $(u_n)_{n \ge 0 }$ be an integral linear recurrence, that is, 
there exists $a_1,a_2,\ldots,a_k \in \ZZ$ with $a_k \ne 0 $ such that 
$$u_n = a_1 u_{n-1} + a_2 u_{n-2} + \cdots + a_k u_{n-k}$$
for all integers $n \ge k.$
The sequence is said to be \textit{nondegenerate} if none of the ratios $\alpha_i/\alpha_j,~i\ne j,$ is a root of unity, where $\alpha_1, \alpha_2,\ldots,\alpha_t$ are all the pairwise distinct zeroes of the \textit{characteristic polynomial} $$\psi_u(x) = x^k - a_1x^{k-1} - a_2 x^{k-2} - \cdots - a_k,$$
The sequence $(u_n)_{n\ge 0}$ is said to be a Lucas sequence if $u_0 = 0, u_1 = 1,$ and $k=2.$ 
Define $\Delta_u \colonequals a_1^2 + 4a_2$ to be the discriminant of the characteristic polynomial of $(u_n)_{n\ge 0}.$ 

Let $g_u(n)$ be the arithmetic function defined by $\gcd(n, u_n).$ 
Several authors have studied the distributional properties of $g_u.$
For instance, the set of all positive integers $n$ such that $u_n$ is divisible by $n$ has been studied by Alba Gonz\'{a}lez, Luca, Pomerance, and Shparlinski in \cite{gonzalez} under the hypothesis that the characteristic polynomial of $(u_n)_{n\ge 0}$ has only simple roots.
The same set was also studied by
Andr\'{e}-Jeannin \cite{jeannin}, Luca and Tron \cite{luca-tron}, Sanna \cite{sanna-ijnt}, and Somer \cite{lucas}, in the special
case in which $(u_n)_{n\ge 0}$ is a Lucas sequence or the Fibonacci sequence. 

On the other hand, Sanna and Tron \cite{sanna-tron-indag, sanna-gcd} have studied the fiber $g_{u}^{-1}(y)$ where $y=1$ and $(u_{n})_{n\ge 0}$ is nondegenerate, and in case $(u_{n})_{n\ge 0}$ is the Fibonacci sequence with $y$ being an arbitrary positive integer. The image $g_{u}(\mathbb{N})$ has been analysed by Leonetti and Sanna \cite{leonetti-sanna} in case $(u_{n})_{n\ge 0}$ is the Fibonacci sequence.
Similar problems, with $(u_n)_{n\ge 0}$ replaced by an elliptic divisibility sequence or by the orbit of $0$ under a polynomial
map, were also studied \cite{chen, gassert, gottschlich, JhaPreprint, kim, silverman-stange}. All recent developments in the study of $g_u$ have been discussed in a recent survey by Tron \cite{tron}.

The previous results give rather convincing answers to the problem of determining  extreme values of $g_{u}(n)$, however, obtaining information about its average behaviour and distribution as arithmetic function has recently got particular interest.
A natural question posed by Sanna in \cite{sanna} is--
\begin{question}[\cite{sanna}]
What is the average value of $g_u$? Or more generally, given a positive integer $\lambda > 0,$ is it possible to find an asymptotic for 
$$\sum_{n\le x} g_u(n)^{\lambda}$$
as $x$ is large?
\end{question}
Hereafter, we assume that $(u_n)_{n\ge 0}$ is a nondegenerate Lucas sequence with $a_1$ and $a_2$ relatively prime integers.
Given the oscillatory behaviour of $g_{u},$ which makes it hard to investigate, the author succeeded in finding an asymptotic for the logarithms of $g_u$. 
\begin{theorem}[\protect{\cite[Theorem 1.1]{sanna}}]\label{thm:sanna-main}
Fix a positive integer $\lambda$ and some $\e > 0.$ 
We have 
$$\sum_{n\le x} (\log g_u(n))^{\lambda} = M_{u,\lambda}\,x + E_{u, \lambda}(x),$$
where $M_{u, \lambda} > 0$ is a constant depending on $a_1, a_2,$ and $\lambda,$ and 
$$E_{u,\lambda}(x) \ll_{u,\lambda} x^{(3\lambda + 1)/(3\lambda + 2) + \e}.$$
\end{theorem}
Furthermore, the author obtained an convergent infinite series for the constant $M_{u,\lambda},$ but before stating it we need to introduce some notations. For each positive integer $m$ relatively prime to $a_2,$ let $z_u(m)$ be the \textit{rank of appearance} of $m$ in the Lucas sequence $(u_n)_{n\ge 0},$ that is, $z_u(m)$ is the least positive integer $n$ such that $m$ divides $u_n.$ It is well known that the rank of appearance exists (see, e.g., \cite{rank} ). Also, define $\l(m)\colonequals \lcm(m, z_u(m)).$
\begin{theorem}[\protect{\cite[Theorem 1.2]{sanna}}]\label{thm:M-series}
For all positive integers $\lambda,$ we have
$$M_{u,\lambda}= \sum_{(m\,,\,a_2)\, =\, 1}\frac{\rho_\lambda(m)}{\l(m)},$$
where the sum runs over all positive integers relatively prime to $a_2.$
\end{theorem}
Recently, Mastrostefano \cite{mastrostefano} obtained a partial answer to the question posed by Sanna \cite{sanna} and proved a nontrivial upper bound on the moments of $g_{u}.$ Given the rich structure of $g_u$, and in order to investigate the relationships between shifted primes and Lucas sequences;
Jha and Sanna in \cite{jha-rnt} studied the set $$\mathcal P_k = \{p\le x : g_u(p-1) = k, ~ p \text{ prime}\}$$ for positive integers $k,$ and proved the existence of relative density of $\mathcal P_k$ in the set of prime numbers and also obtained it as an absolutely convergent series.
Furthermore, they proved bounds on the distribution of positive integers of the form $g_u(p-1)$ for primes $p.$ 
Exploring further in this direction we investigate the average behaviour and distribution of $g_{u}$ under shifted prime arguments. We prove the following theorem concerning the average behaviour of $g_{u}(p-1)$--
\begin{theorem}\label{THM:MAIN}
Fix a positive integer $\lambda$ and some $A, \e > 0.$ 
We have 
$$\sum_{p\le x} (\log g_u(p-1))^{\lambda} = P_{u,\lambda}\,\pi(x) + E_{u, \lambda}(x),$$
where $P_{u, \lambda} > 0$ is a constant depending on $a_1, a_2,$ and $\lambda,$ and 
$$E_{u,\lambda}(x)\ll_{u,\lambda} \frac{x}{(\log x)^A}.$$ Conditional on the Generalized Riemann Hypothesis (GRH), we have
$$E_{u,\lambda}(x) \ll_{u,\lambda} x^{(6\lambda + 3)/(6 \lambda + 4) + \e}.$$
\end{theorem}

\begin{theorem}\label{THM:P-SERIES}
For all positive integers $\lambda,$ we have
$$P_{u,\lambda}= \sum_{(n\, ,\, a_2) \,=\, 1}\frac{\rho_\lambda(n)}{\varphi(\l(n))},$$
where the sum runs over all positive integers relatively prime to $a_2.$
\end{theorem}
As remarked by Sanna \cite{sanna} and Tron \cite{tron}, Theorems \ref{thm:sanna-main}  and \ref{thm:M-series} bear a formal resemblance with work of Luca and Shparlinski \cite[Theorem 2]{luca}. The authors studied sums of the form $\sum_{n\le x}f(u_{n})^{k}$ for arbitrary arithmetic functions $f$ satisfying some growth conditions, and obtained asymptotics of the form $\sum_{n\le x}f(u_{n})^{k}\sim M_{f,k}x$. They also pointed out
that $\log M_{f,k}\ll k \log k$. Motivated by these results, we obtain estimates of the constants $M_{u,\lambda}$ and $P_{u,\lambda}$ of Theorems \ref{thm:M-series} and \ref{THM:P-SERIES}, respectively.
\begin{theorem}\label{THM:P-ASY}
For each positive integer $\lambda,$ we have 
\begin{enumerate}
    \item $\log M_{u,\lambda} = \lambda\log\lambda + O_u(\lambda),$ where $M_{u,\lambda}$ is defined in Theorem \ref{thm:M-series}.
    \item $\log P_{u,\lambda} = \lambda\log\lambda + O_u(\lambda),$ where $P_{u,\lambda}$ is defined in Theorem \ref{THM:P-SERIES}.
\end{enumerate}

\end{theorem}

Another important direction to investigate is estimating the distribution function of $g_{u}$.
As an application of Sanna's results in \cite{sanna}, the author obtained an upper bound on the count $\#\{n\le x : g_{u}(n)>y\}$ for all $x,y>1$. Mastrostefano in \cite{mastrostefano} improved these bounds for a specific range of $y$. 
As a corollary of our Theorem \ref{THM:MAIN}, we obtain upper bounds on the count $\#\{p\le x : g_{u}(p-1)>y\}$ for all $x,y>1$. 
\begin{corollary}
For each positive integer $\lambda$, we have 
\begin{equation*}
    \#\{p\le x \,:\,g_{u}(p-1)\,>\,y\}\ll_{u,\lambda} \frac1{(\log y)^{\lambda}}\frac{x}{\log x},
\end{equation*}
for all $x,y>1$.
\end{corollary}

Exploring further in this direction, we prove that for a fixed $y>0$, there exist infinitely many runs of $m$ consecutive primes in short intervals such that $g_{u}(p-1)>y$. The result is essentially based on a recent remarkable framework of Zhang, Maynard, Tao, et al \cite{maynard} on small gaps between primes. In, particular, we require a theorem of Freiberg~\cite[Theorem 1]{Frieberg} on consecutive primes in arithmetic progressions.

\begin{proposition}\label{gap}
Let $p_1=2<p_2=3<\cdots$ be the sequence of prime numbers. Let $a$ and $q$ be a relatively prime pair of integers, and $m\ge 2$ be an integer. For infinitely many $n$, we have 
\begin{equation*}
    p_{n+1}\equiv \cdots \equiv p_{n+m} \equiv a \Mod{q} \text{ and } p_{n+m}-p_{n+1}\le qB_{m}.
\end{equation*}
It has been shown in \cite[Theorem 1.1]{maynard} that we can take $B_{m}=c\,m^{3}\,e^{4m}$ for an absolute and effective constant $c$.
\end{proposition}
The above proposition gives the following immediate corollary.
\begin{corollary}
For positive integers $m$ and $y$, we have infinitely many runs of $m$ consecutive primes such that $g_u(p_{n+i}-1)>y$ for each $0\le i \le m-1$ and 
\begin{equation*}
    \liminf_{n\,\rightarrow\,\infty}\, (p_{n+m-1}-p_{n})<C\,y^{2}\,m^{3}e^{4m},
\end{equation*}
where $p_{n}$ denotes the $n^{th}$ prime and $C$ is an absolute constant. Moreover, if $\Delta_{u} \ne 1$, we have 
\begin{equation*}
    \liminf_{n\,\rightarrow\,\infty}\, (p_{n+m-1}-p_{n}) \le C\,y\,m^{3}e^{4m},
\end{equation*}
 and in case $m=2$, we have 
\begin{equation*}
    \liminf_{n\,\rightarrow\,\infty}\, (p_{n+1}-p_{n}) \le 246y.
\end{equation*}
\end{corollary}

\begin{proof}
    Choose a positive integer $s\in (y, 2y]$ and consider the arithmetic progression $(1 + \ell(s)n)_{n = 1}^{\infty}$. Noting that $\l(n)\le 2n^{2}$ (see Lemma \ref{lem:basic}-(\ref{ite:basic:upper})) and applying Proposition ~\ref{gap}, we get the first result.
For the second part, as $\Delta_{u}\ne \pm1$, we have a prime $p\mid \Delta_{u}$ such that $\l(p^{r})=p^{r}$ for any positive integer $r$. Consider $g=p^{\lceil{\log_{p}y}\rceil}$ and the arithmetic progression $(1 + gn)_{n = 1}^{\infty}$. Applying Proposition~\ref{gap} again, we get that there exist infinitely many runs $p_n, p_{n+1},\ldots, p_{n+m-1}$ of $m$ consecutive primes such  that $p_{n+m-1}-p_{n}\le C\,y\,m^{3}e^{4m}$ and $g_u(p_{n+i} - 1)>y,~ 0\le i \le m-1$. In the case of $m=2$, it has been obtained that $B_{2}=246$ works in \cite[Theorem 3.2]{polymath}, which gives us the last assertion.
\end{proof}

Having considered the distribution of $g_{u}$, it is natural to consider the problem of analysing the growths of $\max\,\{g_{u}(n)\,:\,n\le x\}$ and $\max\,\{g_{u}(p-1)\,:\,n\le x\}.$ 
For the remainder of the paper, we assume that $\Delta_u\ne 1$. 

It is easy to prove that
$$x\ll\max\,\{g_{u}(n)\,:\,n\le x\}\le x.$$
To see this, we have the trivial upper bound $\max\,\{g_{u}(n)\,:\,n\le x\} \le x$. Since $\Delta_u\ne \pm\,1,$ we have a prime $p\mid \Delta_{u}$ such that $\l(p^{r})=p^{r}$ for any positive integer $r.$ Thus, we get that $\max\,\{g_{u}(n)\,:\,n\le x\}\ge p^{\lfloor{\log _{p} x \rfloor}} \ge x/p$. This gives us desired conclusion. 

One can observe that the problem of obtaining estimates for $\max\,\{g_{u}(n)\,:\,n\le x\}$ is related to the study of positive integers $n$ such that $n\mid u_n$. 
If we were able to demonstrate that for any sufficiently large $x$, there is an integer $n\in (x-o(x),x)$ such that $n\mid u_n,$ we could then prove $\max\,\{g_{u}(n)\,:\,n\le x\}\,\sim\, x$.
If we consider the case of Fibonacci sequence $(F_n)_{n\ge 0}$, then it is known that at least $x^{1/4}$ positive integers $n$ less than $x$ satisfy $n\mid F_{n}$ \cite[Theorem 1.3]{gonzalez}, which is similar to the lower bound of $x^{1/3}$ obtained for Carmichael numbers less than $x$ \cite{harman}.
Luca and Tron \cite{luca-tron} pointed out that one should expect heuristics for self-Fibonacci divisors to be similar to those for Carmichael numbers. 
Larson \cite{carmichael} recently proved that for all $\delta>0$ and $x\gg_{\delta} 1$, there exist at least $e^{\log x/(\log\log x)^{2+\delta}}$ Carmichael numbers between $x$ and $x+x/(\log x)^{1/2+\delta}$. 
Thus, it is reasonable to expect similar results to hold for positive integers $n$ dividing $u_n$.
Based on this observation, we make the following conjecture--

\begin{conjecture}
For Lucas sequences $(u_n)_{n \ge 0}$ with $\Delta_u \ne 1,$
$$\max\,\{g_{u}(n)\,:\,n\le x\}\sim x.$$
\end{conjecture}

Finding nontrivial lower bounds on the shifted prime analogue $\max\,\{g_{u}(p-1)\,:\,p\le x\}$ is notably more difficult. 
For the ease of notation, let us set $$\mathcal G(x)\,\colonequals \,\max\,\{g_{u}(p-1)\,:\,p\le x\}.$$
Our next theorem is in this direction.
\begin{theorem}\label{THM:MAXBOUND}
Let $(u_{n})_{n\ge 0}$ be a Lucas sequence such that $\Delta_u\ne 1.$ Then we have that $\mathcal{G}(x)\gg x^{0.4736}$ unconditionally, while under Montgomery's conjecture (see Hypothesis ~\ref{hypo:mont}) or Chowla's conjecture (see ~\cite{chowla1934}), we obtain that $\mathcal{G}(x)\gg x^{1-o(1)}$.
\end{theorem}

The paper is organized as follows. In Section \ref{sec:2}, we state some well-known results and prove preliminary lemmas. In Section \ref{sec:3}, we prove Theorems \ref{THM:MAIN} and \ref{THM:P-SERIES}. Proofs of Theorems \ref{THM:P-ASY} and \ref{THM:MAXBOUND} are presented in Sections \ref{sec:4}, and \ref{sec:5}, respectively.

\subsection*{Notations}
We employ the Landau-Bachman ``Big Oh'' and ``little oh'' notations $O$ and $o,$ as well as the associated Vinogradov symbols $\ll$ and $\gg,$ with their usual meanings.
Any dependence of implied constants is explicitly stated or indicated with subscripts. Notations like $O_u$ and $o_u$ are shortcuts for $O_{a_1, a_2}$ and $o_{a_1,a_2}, $ respectively.
Throughout, the letters $p$ and $q$ reserved for prime numbers. We write $(a,b)$ or $\gcd(a,b)$ to denote the greatest common divisor of $a$ and $b,$ and $[a,b]$ or $\lcm(a,b)$ to denote the least common multiple of the same.
As usual, denote by $\tau(n),$ $\omega(n),$ and $P(n),$ for the number of divisors, the number of prime factors, and the greatest prime factor, of a positive integer $n,$ respectively.

For every $x > 0$ and for all integers $a$ and $b$, let $\pi(x;b,a)$ be the number of primes $p \leq x$ such that $p \equiv a \Mod b.$ Also denote the error in prime number theorem as
\begin{equation*}
\Delta(x; b, a) \colonequals  \pi(x;b,a) - \frac{\pi(x)}{\varphi(b)}.
\end{equation*}
Lastly, the incomplete gamma function $\Gamma$ is defined as
$$\Gamma(s,x) \colonequals \int_{x}^{\infty}t^{s-1}e^{-t}\mathrm d t=\int_{e^x}^{\infty}\frac{(\log t)^{s + 1}}{t^2}\mathrm d t.$$

\section{Lemmas and Preliminaries}\label{sec:2}
In what follows, let $(u_n)_{n\ge 0}$ be a nondegenerate Lucas sequence with $\gcd(a_1,a_2) = 1.$ 
Note that the discriminant $\Delta_u \ne 0$ as $(u_n)_{n\ge 0}$ is nondegenerate. 

\begin{lemma}\label{lem:basic}
For all positive integers $m,n,j$ and and for all prime numbers $p\nmid a_2,$ we have:
\begin{enumerate}
    \item  \label{ite: gcd} $m\mid g_u(n)$ if and only if $(m,a_2) = 1$ and $\l(m) \mid n.$
    \item \label{ite:lcm}$\lcm(\l(m),\l(n)) = \l(\lcm(m,n)),$ whenever $(mn,a_2) = 1.$
    \item \label{ite:ell}$\l(p^j) = p^j z_u(p)$ if $p\nmid \Delta_u,$ and $\l(p^j) = p^j$ if $p\mid \Delta_u.$
    \item \label{ite:basic:upper} $\l(n)\leq 2n^{2}.$ 
\end{enumerate}
\end{lemma}
\begin{proof}
See \cite[Lemma 2.1]{sanna} for facts (1)-(3). The last fact follows directly from the well-known inequality $z_{u}(n) \leq 2 n$~(see, e.g., \cite{somer}).
\end{proof}

For each positive integer $\lambda$ and for each positive integer $n>1$ with prime factorisation $n=q_{1}^{h_{1}}\cdots q_{s}^{h_{s}}$, where $q_{1}<\cdots<q_{s}$ are prime numbers and $h_{1},\ldots,h_{s}$ are positive integers, define 
\begin{equation*}
    \rho_{\lambda}(n)\colonequals  \lambda! \sum_{\lambda_{1}\, +\,\cdots\,+\,\lambda_{s}\,=\,\lambda} \prod_{i=1}^{s}\frac{(h_{i}^{\lambda_{i}}-(h_{i}-1)^{\lambda_{i}})(\log q_{i})^{\lambda_{i}}}{\lambda_{i}!}
\end{equation*}
where sum is extended over all the $s$-tuples $(s\ge 1)$ of positive integers $(\lambda_{1},\ldots,\lambda_{s})$ such that $\lambda_{1}+\cdots+ \lambda_{s}=\lambda$. Note that $\rho_{\lambda}(n)=0$ when $s>\lambda$. For the sake of convenience, we set $\rho_{\lambda}(1)=0$.

The next lemma is a slightly improved upper bound on the arithmetic function $\rho_{\lambda}$ than the one proved in \cite[Lemma 2.4]{sanna}.

\begin{lemma}\label{lm:rho}
For all positive integers $\lambda$ and $n$, we have $\rho_{\lambda}(n)\le (\log n)^{\lambda}$.
\end{lemma}
\begin{proof}
Let $n=q_{1}^{h_{1}}\cdots q_{s}^{h_{s}}$ be the prime factorisation of $n$, with prime numbers $q_1<\cdots <q_{s}$ and positive exponents $h_{1},\ldots,h_{s}$. Assume also that $s\le \lambda,$ since otherwise $\rho_{\lambda}(n)=0.$
Therefore, 
\begin{equation*}
    \rho_{\lambda}(n)\le \sum_{\lambda_{1}\, +\,\cdots\,+\,\lambda_{s}\,=\,\lambda} \frac{\lambda!}{\lambda_{1}!\cdots \lambda_{s}!} \prod_{i=1}^{s}(h_{i}\log q_{i})^{\lambda_{i}}\le (\log n)^{\lambda},
\end{equation*}
by the multinomial theorem.
\end{proof}
\noindent For $x,y>0$ and positive integer $r,$ define 
\begin{equation*}
    \gamma(r)\colonequals \#\{n \in \mathbb{N} : (n,a_2)=1 \text{ and }\l(n)=r \}\quad\text{and}\quad \Phi(x,y)\colonequals \#\{n\le x : P(n)<y\}.
\end{equation*}
\begin{lemma}\label{lprop}
For all positive integers $r$, we have that $\gamma(r)\le\tau(r)$ where $\tau(r)$ denotes the number of divisors of $r.$
\end{lemma}
\begin{proof}
As $n\mid \l(n)=r$, there are at most $\tau(r)$ possible values of $n$ such that $\l(n)=r$.
\end{proof}

\begin{lemma}\label{lem:Phi-bound}
 Let $C>0$ be a constant. For all sufficiently large $x$, we have $\Phi(x,C)\le (2\log x)^C$.
\end{lemma}
\begin{proof}
Each of the positive integers $n$ counted by $\Phi(x,C)$ can be written as $n=p_{1}^{a_1}\cdots p_{\pi(C)}^{a_{\pi(C)}}$ where $p_{1},\ldots,p_{\pi(C)}$ are all prime numbers less than $C$, and $a_{1},\ldots,a_{\pi(C)}$ are non-negative integers. Clearly there are at most $1+\log x/\log 2 $ choices for each $a_{i}$. Therefore,
\begin{equation*}
    \Phi(x,C)\le \bigg(1+\frac{\log x}{\log 2}\bigg)^{C}\le (2\log x)^{C},
\end{equation*}
as desired.
\end{proof}

The next three lemmas are upper bounds for certain sums involving $\l$. 
\begin{lemma}\label{lm: prime}
We have 
\begin{equation*}
    \sum_{\substack{P(n) \,\geq\, y \\[1pt] (n,\, a_{2}) \,=\, 1}}\frac{1}{\l(n)} 
\ll_{u} \frac{1}{y^{1/3-\varepsilon}}
\end{equation*}
for all $\varepsilon\in (0,1/4]$ and $y\gg_{u,\lambda} 1$.
\end{lemma}
\begin{proof}
See \cite[Lemma 2.5]{sanna}.
\end{proof}

\begin{lemma}\label{lem:bound-ell}
We have 
$$\sum_{\substack{n \,\geq\, y \\[1pt] (n,\, a_{2}) \,=\, 1}} \frac{\rho_{\lambda}(n)}{\l(n)} \ll_{u, \lambda} \frac{1}{y^{1/(1 + 3\lambda) - \e}}$$
for all positive integers $\lambda,$ $\e \in (0, 1/5],$ and $y\gg_{u,\lambda, \e} 1.$
\end{lemma}
\begin{proof}
See \cite[Lemma 2.6]{sanna}.
\end{proof}

\begin{lemma}\label{lem:bound-phi-ell}
We have
\begin{equation*}
\sum_{\substack{n \,\geq\, y \\[1pt] (n,\, a_{2}) \,=\, 1}}\frac{\rho_{\lambda}(n)}{\varphi\big(\ell_{u}(n)\big)} 
\ll_{u,\lambda} \frac{\log \log y}{y^{1/(1+3\lambda)-\varepsilon}} ,
\end{equation*} 
for all positive integers $\lambda,\varepsilon \in (0,1/5],$ and  $y\gg_{u,\lambda,\e} 1$.
\end{lemma}
\begin{proof}
From Lemma \ref{lem:bound-ell}, it follows that
\begin{equation*}
S(t) \colonequals \sum_{\substack{n \,\geq\, t \\[1pt] (n,\, a_{2}) \,=\, 1}}\frac{\rho_{\lambda}(n)}{\ell_{u}(n)} 
\ll_{u,\lambda} \frac{1}{t^{1/(1+3\lambda)-\varepsilon}}
\end{equation*}
for all $y \gg_{u,\lambda,\varepsilon} 1$.
By partial summation,
\begin{align*}\label{equi}
\sum_{\substack{n \,\geq\, y \\[1pt] (n,\, a_{2}) \,=\, 1}}\frac{\rho_{\lambda}(n)\log \log n}{\ell_{u}(n)}  &= S(y)\log \log y +\int_{y}^{+\infty} \frac{S(t)}{t \log t}\,\mathrm{d} t \\
&\ll_{u,\lambda} \frac{\log \log y}{y^{1/(1+3\lambda)-\varepsilon}} + \int_y^{+\infty} \frac{1}{t^{1+1/(1+3\lambda)-\varepsilon} \log t}\,\mathrm{d} t \\
&\ll \frac{\log \log y}{y^{1/(1+3\lambda)-\varepsilon}}.
\end{align*}
Since $\varphi(n)\gg n/\!\log{\log{n}}$ (see, e.g., \cite[Chapter I.5, Theorem 4]{tenebaum}) and $\l(n)\leq 2n^2$ (see Lemma \ref{lem:basic}-(\ref{ite:basic:upper})) for all positive integers $n$, it follows that
\begin{equation*}
\sum_{\substack{n \,\geq\, y \\[1pt] (n,\, a_{2}) \,=\, 1}}\frac{\rho_{\lambda}(n)}{\varphi\big(\ell_{u}(n)\big)}  \ll \sum_{\substack{n \,\geq\, y \\[1pt] (n,\, a_{2}) \,=\, 1}}\frac{\rho_{\lambda}(n) \log \log n}{\ell_{u}(n)} \ll  \frac{\log \log y}{y^{1/(1+3\lambda)-\varepsilon}}.
\end{equation*}
\end{proof}

\vspace{0.1cm}

\begin{lemma}[Siegel-Walfisz under GRH]\label{thm:Siegel-Walfisz}
Under the Generalized Riemann Hypothesis (GRH), we have 
\begin{equation*}
\Delta(x; b, a) \ll x^{1/2}(\log x) ,
\end{equation*}
for all $x \gg 1$.
\end{lemma}
\begin{proof}
See~\cite[Corollary 13.8]{mont}.
\end{proof}

\begin{hypothesis}[Montgomery's conjecture, \protect{\cite[Conjecture 1(b)]{lin}}]\label{hypo:mont}
For any $\varepsilon>0$, there exists a constant $C_\varepsilon$ such that for all $b\le x$ we have
\begin{equation*}
\Delta(x; b, a) \le C_\varepsilon\, {x^{1/2+\varepsilon}}\,{b^{1/2}} ,
\end{equation*}
for all $x \gg 1$.
\end{hypothesis}

\begin{lemma}[Bombieri-Vinogradov]\label{thm:bomb}
For any $A > 0$, there exists $B=B(A)>0$ such that 
\begin{equation*}
\sum_{d\leq \sqrt{x}/(\log x)^{B}} |\Delta(x; d, a)| \ll \frac{x}{(\log x)^{A}} ,
\end{equation*}
for all $x \gg_A 1$.
\end{lemma}
For our applications we need the following weighted version of the Bombieri-Vinogradov theorem.
\begin{lemma}\label{thm:bomb-modify}
For any $A > 0$, there exists $B=B(A)>0$ such that 
\begin{equation*}
\sum_{d\leq \sqrt{x}/(\log x)^{B}} \gamma(d)  |\Delta(x; d, a)| \ll \frac{x}{(\log x)^{A}} ,
\end{equation*}
for all $x \gg_A 1$.
\end{lemma}
\begin{proof}
Replace $A$ with $2A+4$ in Lemma~\ref{thm:bomb} to get a constant $B = B(2A+4).$ Since $\pi(x;d,a)\ll x/d$ for $d\le x$, we have that 
$$\Delta(x; b, a)\ll \frac{x}{d}.$$
Therefore, by putting $z=x^{1/2}/(\log x)^{B}$, we get
\begin{align*}
\sum_{d\leq z} \gamma(d) |\Delta(x; d, a)| &\ll  \sum_{d\leq z} \gamma(d) \left(\frac{x}{d}\right)^{1/2}  |\Delta(x; d, a)|^{1/2}\\
& \ll x^{1/2} \left(\sum_{d\le z}\frac{\gamma(d)^{\,2}}{d}\right)^{1/2}\bigg(\sum_{d\le z}|\Delta(x; d, a)|\bigg)^{1/2}
\end{align*}
\noindent by the Cauchy-Schwarz inequality. By Lemma ~\ref{thm:bomb} and the choice of $2A+4,$ we observe that 
\begin{equation*}
\bigg(\sum_{d\le z}|\Delta(x; d, a)|\bigg)^{1/2}\ll \frac{x^{1/2}}{(\log x)^{A+2}}, 
\end{equation*}
 and by Lemma ~\ref{lprop}, that 
\begin{equation*}
\bigg(\sum_{d\le z}\frac{\gamma(d)^{\,2}}{d}\bigg)^{1/2}\ll \bigg(\sum_{d\le z}\frac{\tau(d)^{\,2}}{d}\bigg)^{1/2} \ll (\log z)^{2} \ll (\log x)^{2},   
\end{equation*}
where the second inequality follows from the fact that \begin{equation*}
    \sum_{n\le x}\tau(n)^{2}=C_{2}x(\log x)^{3}+O(x(\log x)^{2})
\end{equation*}
for some constant $C_2 > 0$ (see, e.g., \cite[Theorem 1]{luca2017}) and partial summation.
\end{proof}
\begin{remark}
It is worth mentioning that the weighted Bombieri-Vinogradov theorem obtained above can be used in the proof of \cite[Theorem 1.2]{jha-rnt} to improve the error term of the relative density estimate obtained. The authors originally obtained an error term of the form
\begin{equation*}
    E(x)\ll \frac{x}{\log x}\cdot \frac{\log \log \log x}{\exp\big(\delta (\log \log x)^{1/2}(\log \log \log x)^{1/2}\big)},
\end{equation*}
for some $\delta>0,$ while employing Lemma \ref{thm:bomb-modify} instead leads to an effective upper bound on the error term of the form 
\begin{equation*}
    E(x)\ll \frac{x}{(\log x)^{A}}
\end{equation*}
for any $A>0$.
\end{remark}

\section{Proofs of Theorems \ref{THM:MAIN} and \ref{THM:P-SERIES}}\label{sec:3}
Throughout the section, the letter $q$, with or without subscript, denotes a prime number not dividing $a_2$, while the letter $j$, with or without subscript, denotes a positive integer. 
We have
\begin{equation*}
    \log g_{u}(p-1)=\sum_{q^{j} \, \mid\mid\, g_{u}(p-1)} j\log p=\sum_{q^{j} \, \mid \, g_{u}(p-1)}\log p=\sum_{p\, \equiv 1 \Mod{\l(q^{j})}}\log p
\end{equation*}
for all primes $p$ where the last equality follows from Lemma ~\ref{lem:basic}-(\ref{ite: gcd}).
Consequently, for any prime $p$ and for all $x>0$, it follows that
\begin{align*}\label{equi}
\sum_{p \,\leq\, x}(\log g_{u}(p-1))^{\lambda} &= \sum_{p \,\leq\, x}\bigg(\sum_{p \, \equiv 1 \Mod{\l(q^{j})}}\log q \bigg)^{\lambda} \\
&=\sum_{p \,\leq\, x} \sum_{\l(q_{1}^{j_1}) \, \mid \, p-1, \, \ldots ,\, \l(q_{\lambda}^{j_{\lambda}}) \, \mid \, p-1} \log q_{1} \cdots \log q_{\lambda} \\
&= \sum_{p \,\leq\, x} \sum_{p\, \equiv \,1 \Mod{\l([q_{1}^{j_{1}},\, \ldots,\, q_{\lambda}^{j_{\lambda}}]}} \log q_{1} \cdots \log q_{\lambda}\\
&=\sum_{q_{1}^{j_{1}}, \, \ldots,\, q_{\lambda}^{j_{\lambda}}} \log q_{1} \cdots \log q_{\lambda} \sum_{\substack{p \,\leq\, x \\ {p\, \equiv \,1 \Mod{\l([q_{1}^{j_{1}},\, \ldots,\, q_{\lambda}^{j_{\lambda}}]}}}} 1 \\&= \sum_{q_{1}^{j_{1}},\, \ldots,\, q_{\lambda}^{j_{\lambda}}} \log q_{1} \cdots \log q_{\lambda} \, \pi(x; \l(q_{1}^{j_{1}},\ldots,q_{\lambda}^{j_{\lambda}}), 1) \\&= \sum_{(n,\,a_{2})=1} \pi(x; \l(n), 1) \sum_{n=[q_{1}^{j_{1}},\ldots,q_{\lambda}^{j_{\lambda}}]}\log q_{1} \cdots \log q_{\lambda}.
\end{align*}
where third equality follows from Lemma ~\ref{lem:basic}-(\ref{ite:lcm}).
By the exact same reasoning as \cite[Section 3]{sanna} we have
\begin{align*}
    \sum_{n=[q_{1}^{j_{1}},\ldots,q_{\lambda}^{j_{\lambda}}]}\log q_{1} \cdots \log q_{\lambda} 
    &= \sum_{\lambda_1 + \cdots + \lambda_s = \lambda} \frac{\lambda !}{\lambda_1 ! \cdots \lambda_s !} \prod_{i=1}^s (h_i^{\lambda_i} - (h_i-1)^{\lambda_i})(\log q_i)^{\lambda_i}\\
    &= \rho_{\lambda}(n).
\end{align*}
Therefore
\begin{equation}
    \sum_{p \,\leq\, x}(\log g_{u}(p-1))^{\lambda}=\sum_{\substack{d \, \le \, x \\ (d,\, a_{2})\,=\,1}}\rho_{\lambda}(d)\pi(x;\l(d),1)\label{eqn:main}
\end{equation}
for all $x>0$. 
\subsection{Unconditional bounds}
Choose a large constant $A>0$ and $B=B(A)>0$ in the statement of the weighted Bombieri-Vinogradov theorem (Lemma \ref{thm:bomb-modify}). Set $y\, \colonequals\, x^{1/4}/\sqrt{2}(\log x)^{B/2}$ for brevity.
We split the right hand side of (\ref{eqn:main}) as
\begin{align*}
    \sum_{p \,\leq\, x}(\log g_{u}(p-1))^{\lambda}&=\sum_{\substack{d \, \le \, y \\ (d,\,a_{2})\,=\,1}}\rho_{\lambda}(d)\pi(x;\l(d),1)+\sum_{\substack{d \, \ge \, y \\ (d,\,a_{2})\,=\,1}}\rho_{\lambda}(d)\pi(x;\l(d),1)\\&=E_{1}+E_{2}.
\end{align*}
First we estimate $E_{1}$.
\begin{align*}
    E_{1}&=\sum_{\substack{d \, \le \, y \\ (d,\,a_{2})\,=\,1}}\Bigg(\frac{\pi(x)\rho_{\lambda}(d)}{\varphi\big(\l(d)\big)}+\rho_{\lambda}(d)\Delta(x;\l(d),1)\Bigg)\\&=\sum_{\substack{d \, \le \, y \\ (d,\,a_{2})\,=\,1}}\pi(x)\cdot \frac{\rho_{\lambda}(d)}{\varphi\big(\l(d)\big)}+\sum_{\substack{d \, \le \, y \\ (d,\,a_{2})\,=\,1}}\rho_{\lambda}(d) \Delta(x;\l(d),1).\numberthis\label{eqn:E1}
\end{align*}
As $y\rightarrow\infty$, the left sum in (\ref{eqn:E1}) converges to $P_{u,\lambda}\pi(x)$ where 
\begin{equation*}
    P_{u,\lambda} \colonequals \sum_{(n,\, a_{2}) \,=\, 1}\frac{\rho_{\lambda}(n)}{\varphi \big(\ell_{u}(n)\big)}.
\end{equation*}
This follows by convergence of the sum in Lemma~\ref{lem:bound-phi-ell}. 
For the right sum in (\ref{eqn:E1}), if $d\, \le \, y$, then $\l(d) \, \le \, x^{1/2}/(\log x)^{B}$ by Lemma \ref{lem:basic}-(\ref{ite:basic:upper}).
Therefore, by Lemmas~\ref{lm:rho} and \ref{thm:bomb-modify}, 
\begin{align*}
   \sum_{\substack{d \, \le \, y \\ (d,\,a_{2})\,=\,1}}\rho_{\lambda}(d) \Delta(x;\l(d),1) &\le (\log x)^{\lambda}  \sum_{\substack{d \, \le \, y \\ (d,\,a_{2})\,=\,1}}|\Delta(x;\l(d),1)| \\& \le(\log x)^{\lambda} \sum_{d\,\leq\, \sqrt{x}/(\log x)^{B}} \gamma(d)\,|\Delta(x; d, 1)|\\&\ll\frac{x}{(\log x)^{A-\lambda}}.
\end{align*}
For $E_{2},$ we employ the trivial bound $\pi(x;b,a)\leq x/b$ and Lemma \ref{lem:bound-ell} as follows--
\begin{align*}
E_{2}=\sum_{\substack{d \, \ge \, y \\ (d,\,a_{2})\,=\,1}}\rho_{\lambda}(d)\pi(x;\l(d),1)\le x\sum_{\substack{d \, \ge \, y \\ (d,\,a_{2})\,=\,1}}\frac{\rho_{\lambda}(d)}{\l(d)}\ll_{u,\lambda} \frac{x}{y^{1/(1+3\lambda)-\e}}
\end{align*}
where $\varepsilon \in (0,1/5]$.
Thus, finally we obtain that
\begin{equation*}
    \sum_{p \,\leq\, x}(\log g_{u}(p-1))^{\lambda}=P_{u,\lambda}\pi(x)+E_{u,\lambda}(x)
\end{equation*} 
where 
\begin{equation*}
    E_{u,\lambda}(x)\ll \frac{x}{(\log x)^{A}}
\end{equation*}
for any $A>0.$

\subsection{Bounds conditional on GRH}
Let $z$ be a parameter to be chosen later.
Just as before, we split (\ref{eqn:main}) as
\begin{align*}
    \sum_{p \,\leq\, x}(\log g_{u}(p-1))^{\lambda}&=\sum_{\substack{d \, \le \, z \\ (d,\,a_{2})\,=\,1}}\rho_{\lambda}(d)\pi(x;\l(d),1)+\sum_{\substack{d \, \ge \, z \\ (d,\,a_{2})\,=\,1}}\rho_{\lambda}(d)\pi(x;\l(d),1)\\&=E'_{1}+E'_{2}.
\end{align*}
First we look at $E'_{1}$. 
Proceeding as before we just need to look at the sum 
\begin{equation*}
    \sum_{\substack{d \, \le \, z \\ (d,\,a_{2})\,=\,1}}\rho_{\lambda}(d) \Delta(x;\l(d),1).
\end{equation*}
Again, we obtain that 
\begin{equation*}
    \sum_{\substack{d \, \le \, z \\ (d,\,a_{2})\,=\,1}}\rho_{\lambda}(d) \Delta(x;\l(d),1)\ll z\cdot (\log x)^{\lambda}\cdot x^{1/2}(\log x)
\end{equation*}
by Lemmas \ref{lm:rho} and \ref{thm:Siegel-Walfisz}.
Similarly, for $E'_{2}$, we get that 
\begin{equation*}
    E'_{2}\ll_{u,\lambda} \frac{x}{z^{1/(1+3\lambda)-\varepsilon}}.
\end{equation*}
Thus,
\begin{equation*}
    \sum_{p \,\leq\, x}(\log g_{u}(p-1))^{\lambda}=P_{u,\lambda}\pi(x)+E_{u,\lambda}(x),
\end{equation*} 
where 
\begin{equation*}
    E_{u,\lambda}(x)\ll_{u,\lambda} \frac{x}{z^{1/(1+3\lambda)-\varepsilon}}+ z (\log x)^{\lambda}\cdot x^{1/2}(\log x).
\end{equation*}
It is routine to check that the optimal value of $z,$ while satisfying all restrictions, is $z=x^{(3\lambda+1)/(6\lambda+4)}$. In conclusion,
\begin{equation*}
    E_{u,\lambda}(x)\ll_{u,\lambda} x^{(6 \lambda+3)/(6\lambda +4) +\varepsilon}
\end{equation*}
for all sufficiently large $x$ depending on $a_{1}\, ,a_{2},\, \lambda$ and $\varepsilon$.

\section{Proof of Theorem \ref{THM:P-ASY}}\label{sec:4}
It is obvious that $M_{\lambda, u} \le P_{\lambda, u}$ from Theorems \ref{thm:M-series} and \ref{THM:P-SERIES}. So it is sufficient to prove that $\log M_{\lambda, u}\ge \lambda\log\lambda + O(\lambda)$ and  $\log P_{\lambda, u}\le \lambda\log\lambda + O(\lambda).$
We shall use the following basic lemma on the incomplete gamma function to estimate some integrals.
\begin{lemma}\label{lem:incomplete-gamma}
For positive integer $n$ and $x\ge 0$, we have  
$$\Gamma(n,x) = (n-1)!\, e^{-x}\sum_{k=0}^{n-1}\frac{x^k}{k!}$$
\end{lemma}
\begin{proof}
 See \cite[Theorem 3]{jameson_2016}.
\end{proof}
We split the proof into two parts-- the lower bound and the upper bound.
\subsection{Lower bound on $M_{u,\lambda}$} 
Note that $\rho_{\lambda}(m) = (\log m)^{\lambda}$ whenever $m$ is prime. Using Lemma \ref{lem:basic}--(\ref{ite:basic:upper}), it follows that 
$$M_{\lambda,u} = \sum_{(n, a_2) = 1}\frac{\rho_{\lambda}(n)}{\l(n)} \ge \frac12 \sum_{p > a_2}\frac{(\log p)^{\lambda}}{p^2}.$$
By partial summation and the prime number theorem,
\begin{align*}
    \sum_{p>x}\frac{(\log p)^{\lambda}}{p^2} &= -\frac{\pi(x)(\log x)^{\lambda}}{x^2} + \int_{x}^{\infty} \pi(t)\left( \frac{2(\log t)^{\lambda}}{t^3} - \frac{\lambda (\log t)^{\lambda -1}}{t^3}\right)\mathrm d t\\
    &= -\frac{(\log x)^{\lambda -1}}{x} + O\Big(\frac{(\log x)^{\lambda -2}}{x}\Big) - \int_{x}^{\infty}\Big(-\frac{2(\log t)^{\lambda -1}}{t^2} \\ 
    &\quad\, +\frac{(\lambda-2)(\log t)^{\lambda -2}}{t^2}+ O\Big(\frac{\lambda(\log t)^{\lambda -3}}{t^2}\Big)\Big)\mathrm d t\\
    &=  -\frac{(\log x)^{\lambda -1}}{x} + O\left(\frac{(\log x)^{\lambda -2}}{x}\right) + 2\Gamma(\lambda, \log x) \\ 
    &\quad\, - (\lambda -2)\Gamma(\lambda -1, \log x) + O(\lambda\Gamma(\lambda -2, \log x))\\
    &= 2\Gamma(\lambda, \log x)- (\lambda -2)\Gamma(\lambda -1, \log x) + O_x(\lambda(\lambda -3)!).
\end{align*}
The last step follows from Lemma \ref{lem:incomplete-gamma}.
Note that $\Gamma(\lambda, \log x) \ge (1-o_{\lambda}(1))(\lambda - 1)!$
and $\Gamma(\lambda - 1,\log x) \le (\lambda -2)!.$
Therefore, $2\Gamma(\lambda, \log x)- (\lambda -2)\Gamma(\lambda -1, \log x) \ge (1-o_\lambda(1))\lambda(\lambda -2)!$ and we conclude that $M_{\lambda, u} \ge (0.5-o_{\lambda}(1))\lambda(\lambda -2)!,$ completing the proof of the lower bound.

\subsection{Upper bound on $P_{\lambda, u}$}
We show bounds on tail sums of $P_{\lambda, u}$:
\begin{equation*}
    \sum_{\substack{n \,\geq\, y \\[1pt] (n,\, a_{2}) \,=\, 1}}\frac{\rho_{\lambda}(n)}{\varphi\big(\l(n)\big)}
\end{equation*}
Observe that $\rho_\lambda(n) = 0$ when $\omega(n) > \lambda.$
Consider the sum
\begin{equation*}
    \sum_{\substack {(n ,\, a_{2}) \, =\, 1 \\[1pt]  \omega(n) \, \le \, \lambda \\ P(n)\, \ge \, y }}\frac{\rho_{\lambda}(n)}{\varphi\big(\l(n)\big)}
\end{equation*}
and set $y = w^{12}.$
By choosing $\varepsilon=1/4$ in Lemma~\ref{lm: prime} we have that  
\begin{equation*}
    S(w) \colonequals \sum_{\substack {(n,\, a_{2}) \, =\, 1 \\[1pt]  \omega(n) \, \le \, \lambda \\ P(n)\, \ge \, y }}\frac1{\l(n)} \le \frac{C_u}{w}
\end{equation*} 
for all $w\ge w_u$ where $C_u$ and $w_u$ are constants depending on $u$.
Since $\varphi(n)\gg n/\!\log{\log{n}}$ (see, e.g., \cite[Chapter I.5, Theorem 4]{tenebaum}) and $\l(n)\leq 2n^2$ (see Lemma \ref{lem:basic}-(\ref{ite:basic:upper})) for all positive integers $n$, it follows that 
\begin{align*}
    \sum_{\substack {(n,\, a_{2}) \, =\, 1 \\[1pt]  \omega(n) \, \le \, \lambda \\ P(n)\, \ge \, y }}\frac{\rho_{\lambda}(n)}{\varphi\big(\l(n)\big)} 
    &\le \sum_{\substack {(n,\, a_{2}) \, =\, 1 \\[1pt]  \omega(n) \, \le \, \lambda \\ P(n)\, \ge \, y }}\frac{(\log n )^{\lambda
    +1}}{\l(n)} = S(w)(\log w)^{\lambda+1} + (\lambda+1)\int_{w}^{\infty} S(t) \frac{(\log t)^{\lambda}}{t} \mathrm{d} t \\
    &\le C_u\bigg( \frac{(\log w)^{\lambda+1}}{w} + (\lambda+1) \int_{w}^{\infty} \frac{(\log t)^{\lambda }}{t^2} \mathrm{d} t\bigg) \\
    &\le C_u \bigg(\frac{(\log w)^{\lambda+1}}{w} + (\lambda+1) \Gamma(\lambda+1,\log w)\bigg)
\end{align*} 
for all $w\ge w_u$. 
Hence, by substituting $w = w_u$ and $y = y_u\colonequals w_u^{12},$ we obtain that 
\begin{equation*}
    \sum_{\substack {(n ,\, a_{2}) \, =\, 1 \\[1pt]  \omega(n) \, \le \, \lambda \\ P(n)\, \ge \, y_u }}\frac{\rho_{\lambda}(n)}{\varphi\big(\l(n)\big)}\le C_u \bigg( \frac{(\log w_u)^{\lambda+1}}{w_u} +(\lambda+1) \Gamma(\lambda+1,\log w_u )\bigg),
\end{equation*}
 which implies that for all sufficiently large $\lambda,$ it follows that
\begin{equation}\label{eqn:sum1}
    \sum_{\substack {(n ,\, a_{2}) \, =\, 1 \\[1pt]  \omega(n) \, \le \, \lambda \\ P(n)\, \ge \, y_u }}\frac{\rho_{\lambda}(n)}{\varphi\big(\l(n)\big)}\le 2C_u (\lambda+1)\Gamma(\lambda+1,\log w_u ) \le 2C_u(\lambda + 1)!.
\end{equation}
Now we focus on the sum 
\begin{equation*}
    \sum_{\substack {(n ,\, a_{2}) \, =\, 1 \\[1pt]  \omega(n) \, \le \, \lambda \\ P(n)\, \le \, y_u \\ n\,>\,w_u }}\frac{\rho_{\lambda}(n)}{\varphi\big(\l(n)\big)}.
\end{equation*}
Using Lemma \ref{lem:Phi-bound}, for all sufficiently large $\lambda,$ we have that 
\begin{align*}
    \sum_{\substack {(n ,\, a_{2}) \, =\, 1 \\[1pt]  \omega(n) \, \le \, \lambda \\ P(n)\, \le \, y_u \\n \, > \, w_u}}\frac{\rho_{\lambda}(n)}{\varphi\big(\l(n)\big)} &\le \sum_{\substack {P(n) \,\leq\, y_u \\[1pt]  \omega(n) \, \le \, \lambda \\ n \, > \, w_u}}\frac{(\log n)^{\lambda+1}}{n}\\
    &=  -\frac{\Phi(w_u,y_u)(\log w_u)^{\lambda+1}}{w_u}+\int_{w_u}^{\infty}\frac{\Phi(t,y_u) ((\log t)^{\lambda+1}-(\lambda+1) (\log t)^{\lambda})}{t^2} \mathrm{d} t \\
    & \ll  \int_{w_u}^{\infty} \frac{(\log t)^{\lambda+1+y_u}}{t^{2}} \mathrm{d}t \\
    &= \Gamma(\lambda +1+y_u, \log w_u) \\
    &\le (\lambda + \lceil y_u\rceil)! \numberthis \label{eqn:sum2}
\end{align*}
Combining (\ref{eqn:sum1}) and (\ref{eqn:sum2}), we we obtain the desired upper bound.

 \section{Proof of Theorem \ref{THM:MAXBOUND}}\label{sec:5}
 In what follows, for relatively prime positive integers $a$ and $d$, $p(a,d)$ denotes the least prime in the arithmetic progression $(a + dn)_{n = 1}^{\infty}$. The remainder of this section depends heavily on different variants of Linnik's theorem on least primes in arithmetic progressions.
 If $n$ divides $g_{u}(p-1),$ we have $p\equiv 1 \Mod{\l(n)}$ (see Lemma ~\ref{lem:basic}-(\ref{ite: gcd})). Henceforth, we need to look for the largest $n$ such that $p(1,\l(n))\le x$.
 Since $|\Delta_{u}|\ne 1,$ it follows from Lemma ~\ref{lem:basic}-(\ref{ite:ell}) that there exists a prime $p\mid \Delta_{u}$ such that $\l(p^{r})=p^{r}$ for any positive integer $r.$ Unconditionally, the best known bound on $p(a,d)$ is of the form $p(a,d) \ll d^{\,2.1115}$ when $d$ varies over powers of a fixed prime number $q$ (see \cite[Theorem 3.6]{li}). We solve for the maximal $n$ such that
  \begin{equation*}
     p(1,\l(q^{n})) \ll (\l(q^{n}))^{\,2.1115}=q^{ 2.1115n} \le x,
 \end{equation*}
 which gives us the lower bound 
 \begin{equation*}
    \mathcal{G}(x)\gg x^{0.4736}.
 \end{equation*}
By an application of Hypothesis ~\ref{hypo:mont}, one gets the folklore conjecture of Chowla ~\cite{chowla1934} that $p(a,d)\ll_{\varepsilon} d^{\,1+\varepsilon}$ for any $\varepsilon>0$. Thus, applying the same technique as before, conditionally, we solve for maximal $n$ satisfying  
\begin{equation*}
     p(1,\l(q^{n})) \ll_{\varepsilon} q^{n(1+\varepsilon)} \le x 
 \end{equation*}
 giving us the lower bound
 \begin{equation*}
    \mathcal{G}(x) \gg x^{1-o(1)}.
 \end{equation*}
\begin{remark}
If we consider the case of Lucas sequence $(u_{n})_{n\ge 0}$ with $\Delta_{u}=1$, then combining the best known upper bound on $p(a,d)$ of the form $p(a,d)\ll d^{\, 5}$ (see \cite[Theorem 2.1]{linnik}) with the fact $\l(n)\le 2n^{2}$ (see Lemma \ref{lem:basic}-(\ref{ite:basic:upper})) we get that $\mathcal{G}(x)\gg x^{0.1}$ unconditionally. Conditional on Montgomery's conjecture (Hypothesis \ref{hypo:mont}), we can improve this bound to $\mathcal{G}(x)\gg x^{0.5 -o(1)}$.
\end{remark}

\let\bfseries\relax
{\setstretch{1}
 \printbibliography
}

\bigskip Abhishek Jha,

Indraprastha Institute of Information Technology, New Delhi, India

Email: \emaillink{abhishek20553@iiitd.ac.in}

\bigskip Ayan Nath,

Chennai Mathematical Institute, Siruseri, Tamil Nadu, India

Email: \emaillink{ayannath@cmi.ac.in}

\end{document}

%% file: defa.tex
\usepackage{fancyhdr}
\contentsmargin{2.0em}
\dottedcontents{section}[2.3em]{}{1.8em}{1pc}
\dottedcontents{subsection}[6.0em]{}{1.8em}{1pc}
\titleformat*{\subsection}{\large\scshape}
\titleformat*{\subsubsection}{\slshape}

\usepackage{titlesec}                                                                            
\titleformat*{\section}{\LARGE\bfseries}
\titleformat*{\subsection}{\Large\itshape}
\titleformat*{\subsubsection}{\scshape}
\titleformat*{\paragraph}{\itshape}

\usepackage{enumitem}
\setitemize[1]{label=$\diamond$\hspace{0.07in}}
\setlist{nolistsep}

\usepackage[nottoc]{tocbibind}

\usepackage{array}

\usepackage{microtype}

\usepackage{colonequals}

\usepackage{tikz}
\usetikzlibrary{matrix,arrows,positioning,decorations.pathmorphing}
\usepackage{tikz-cd}
 
\usepackage{wrapfig}

\usepackage{etex}

\usepackage{url}
\usepackage{verbatim}

\usepackage[notcite,notref]{showkeys}

\usepackage{mathtools}

\usepackage[overload]{textcase}

\newcommand{\eb}[1]{{\itshape\bfseries#1}}
\renewcommand{\emph}{\eb}

\hypersetup{linkcolor=blue,anchorcolor=blue,citecolor=blue,filecolor=blue,urlcolor=blue}


%% file: defs.tex
\newcommand{\bcomment}{}
\newcommand{\bfootnotesize}{\begin{footnotesize}}\newcommand\efootnotesize{\end{footnotesize}}
\newcommand{\bquote}{\begin{quote}}\newcommand\equote{\end{quote}}
\newcommand{\bsmall}{\begin{small}}\newcommand\esmall{\end{small}}
\newcommand{\btable}{\begin{table}}\newcommand{\etable}{\end{table}}

\newcommand{\edocument}{\end{document}}

\newcommand{\tableoc}{\tableofcontents}

\def\1{{1\mkern-7mu1}}

\newcommand{\Mod}{\mathsf{Mod}}



%% file: defs-extra.tex
\renewcommand{\l}{\ell_{u}}
\renewcommand{\ge}{\geqslant}
\renewcommand{\geq}{\geqslant}
\renewcommand{\le}{\leqslant}
\renewcommand{\leq}{\leqslant}

\newcommand{\e}{\varepsilon}
\newcommand{\lcm}{\operatorname{lcm}}
\renewcommand{\Mod}[1]{\ (\mathrm{mod}\ #1)}

\newcommand{\ZZ}{\mathbb Z}

\newcommand\numberthis{\refstepcounter{equation}\tag{\theequation}}
\newcommand{\emaillink}[1]{\textcolor{blue}{\href{mailto:#1}{#1}}}